\documentclass[12pt]{article}
\usepackage[left=2.5cm,right=2.5cm,top=2.5cm,bottom=2.5cm,a4paper]{geometry}

\usepackage[a4paper]{geometry}
\usepackage{graphicx}
\usepackage{microtype}
\usepackage{siunitx}
\usepackage{booktabs}
\usepackage{graphics}
\usepackage{graphicx}
\usepackage{epsfig}
\usepackage{amsmath,amsfonts,amssymb,amsthm}
\usepackage{cleveref}
\usepackage{listings}
\usepackage{paralist}
\usepackage{sectsty}
\usepackage{datetime}
\usepackage{algorithmic,algorithm}
\usepackage{multirow}

\usepackage[X2,T1]{fontenc} 

\usepackage{authblk}

\author[1,\thanks{khsong@jejunu.ac.kr}]{Kyunghwan Song}
\affil[1]{Department of Mathematics, Jeju National University, 102 Jejudaehakro Jeju, 63243, Republic of Korea}

\date{}

\DeclareTextSymbolDefault{\CYRABHDZE}{X2}
\DeclareTextSymbolDefault{\cyrabhdze}{X2}

\numberwithin{equation}{section}
\subsectionfont{\normalfont}

\newtheorem{theorem}{Theorem}[section]
\newtheorem{definition}[theorem]{Definition}
\newtheorem{lemma}[theorem]{Lemma}

\newtheorem{remark}[theorem]{Remark}

\newtheorem{proposition}[theorem]{Proposition}
\newtheorem{conjecture}[theorem]{Conjecture}

\newcommand{\be}{\begin{equation}}
	\newcommand{\ee}{\end{equation}}
\newcommand{\bee}{\begin{equation*}}
	\newcommand{\eee}{\end{equation*}}

\usepackage{color}

\title{
The Frobenius problem for shifted square sequences
}

\begin{document}
	
	\maketitle

\begin{abstract}
	The greatest integer that does not belong to a numerical semigroup $S$ is called the Frobenius number of $S$, and finding the Frobenius number is called the Frobenius problem.  In this paper,  we resolve the conjecture of Frobenius problem for shifted square sequences suggested by Liu and Xin.
\end{abstract}

\section{Introduction}\label{sec_Introduction}
To begin, we introduce a numerical semigroup and submonoid generated by a nonempty subset.
\begin{definition} 
	A {\em numerical semigroup} is a subset $S$ of $\mathbb{Z}_{\geq 0}$ that is closed under addition and contains $0$, such that $\mathbb{Z}_{\geq 0}\texttt{\char`\\}S$ is finite. 
\end{definition}
\begin{definition} \label{def_submonoid}
	Given a nonempty subset $A$ of a numerical semigroup $\mathbb{Z}_{\geq 0}$, we denote by $\big<A\big>$ \textit{the submonoid of $(\mathbb{Z}_{\geq 0},+)$ generated by $A$}, that is,
	\begin{displaymath}
		\big<A\big> = \{\lambda_1 a_1 + \cdots + \lambda_n a_n | n \in \mathbb{Z}^{+}, a_i \in A, \lambda_i \in \mathbb{Z}_{\geq 0}
	\end{displaymath}
	\textrm{ for all } $i \in \{1,\cdots,n\}\}$.
\end{definition}
In addition, we introduce a theorem directly related to the above definitions.
\begin{theorem} (\cite{Rosales2015,Rosales2009}).
	Let $\big<A\big>$ be the submonoid of $(\mathbb{Z}_{\geq 0},+)$ generated by a nonempty subset $A \subseteq \mathbb{Z}_{\geq 0}$. Then $\big<A\big>$ is a numerical semigroup if and only if $\gcd(A) = 1$.
\end{theorem}

For a positive integer $a$, define 
$$
A(a) = \{a, a+1^2, a+2^2, \ldots\}.
$$
We consider the numerical semigroup generated by $A(a)$:
$$
S(A(a)) = \{\lambda_0a + \lambda_1(a + 1^2) + \lambda_2(a+2^2) + \ldots | \lambda_i \in \mathbb{Z}_{\geq 0}\textrm{ for all }i \in \mathbb{Z}_{\geq 0} \}.
$$
For a numerical semigroup $S$, the Frobenius number of $S$ is defined by
$$
F(S) = \max\left(\mathbb{Z}_{\geq 0}\backslash S\right).
$$
Determining $F(S(A(a)))$ remains an open problem in general \cite{Einstein2007,Liu2026}. 

In this paper, we resolve the conjecture proposed in \cite{Liu2026} and obtain explicit formulas for several infinite families. We mainly use the following results in \cite{Liu2026}.
\begin{definition}
	(\cite{Liu2026}) For $n \in \mathbb{Z}^{+}$, define $\iota(n)$ as
	$$
	\iota(n) = \min\{s | n = a_1^2 + a_2^2 + \cdots + a_s^2, a_i \in \mathbb{Z}^{+}, 1\leq i\leq s\}
	$$
\end{definition}
\begin{lemma}
	(\cite{Hardy2008}) (Lagrange's Four-Square Theorem). Every positive integer can be expressed as the sum of four squares. Immediately, for any $m \in \mathbb{Z}^{+}$, there exist $a,b,c,d \in \mathbb{Z}$ such that $m = a^2 + b^2 + c^2 + d^2$.
\end{lemma}
\begin{lemma}
	(\cite{Hardy2008}) Let $m \in \mathbb{Z}^{+}$. Then there exist $a,b \in \mathbb{Z}$ such that $m = a^2 + b^2$ iff $m = p_1^{e_1}\cdots p_k^{e_k}$ (standard prime factorization), for any $i \in \{1,\ldots, k\}, p_i \equiv 3\pmod{4}$ implies that $2|e_i$.
\end{lemma}
\begin{lemma}
	(\cite{Hardy2008}) Let $m \in \mathbb{Z}^{+}$. Then there exist $a,b,c \in \mathbb{Z}$ such that $m = a^2 + b^2 + c^2$ iff $m \neq 4^r(8t + 7)$ for any $r,t \in \mathbb{Z}_{\geq 0}$.
\end{lemma}
Combining the preceding results, we obtain the following classification.
\begin{theorem}
	(\cite{Liu2026}) Let $m \in \mathbb{Z}^{+}$ with standard prime factorization $m = p_1^{e_1}\cdots p_k^{e_k}$. Then $m$ belongs to exactly one of the following mutually exclusive categories:
	\begin{enumerate}[1.]
		\item $2|e_i$ for any $i$ ($m$ is a perfect square).
		\item There exists $i$ such that $e_i$ is odd. And for any $p_j \equiv 3\pmod{4}$, $2|e_j$ ($m$ is not a perfect square and $m$ is a sum of two squares)
		\item $m = 4^r(8t+7)$ for some $r,t \in \mathbb{Z}^{+}$ ($m \neq a^2+b^2+c^2$ for any $a,b,c \in \mathbb{Z}$).
		\item None of the above conditions hold ($m$ is not perfect square, sum of two squares but sum of three squares)
	\end{enumerate}
	Accordingly, $\iota(m)$ is determined by:
	$$
	\iota(m) = \begin{cases}
		1 & \text{ if }n\text{ is of type 1} \\
		2 & \text{ if }n\text{ is of type 2} \\
		4 & \text{ if }n\text{ is of type 3} \\
		3 & \text{ if }n\text{ is of type 4} 
	\end{cases}
	$$
\end{theorem}
\begin{theorem}\label{thm:Liu}
	(\cite{Liu2026}) Let $A(a) = (a,a+1^2,a+2^2,a+3^2,\ldots)$ and $S(A(a))$ be a numerical semigroup generated by $A(a)$. Suppose that there exists $r$ with $1\leq r\leq a-1$ satisfying $\iota(r) = 4, \iota(a+r) \geq 3,$ and $\iota(2a+r) \geq 2$. Then
	$$
	F(S(A(a))) = 3a + \max\{r | \iota(r) = 4, \iota(a+r) \geq 3, \iota(2a+r) \geq 2\}.
	$$
\end{theorem}
\begin{conjecture}\label{conj:Liu}
	(\cite{Liu2026}) Let $A(a) = (a,a+1^2,a+2^2,a+3^2,\ldots)$. For $a > 30$, For every integer $a>30$, the hypotheses of
	Theorem \ref{thm:Liu} are satisfied.
	Consequently,
	$$
	F(S(A(a))) = 3a + \max\{r |1\leq r\leq a-1, \iota(r) = 4, \iota(a+r) \geq 3, \iota(2a+r) \geq 2\}.
	$$
\end{conjecture}
\section{Main Results}
\subsection{The case for $a \equiv 0\pmod{4}$}
In this section, we consider the case
$4\mid a$. Computational evidence suggests that
$$
\max\{r |1\leq r\leq a-1, \iota(r) = 4, \iota(a+r) \geq 3, \iota(2a+r) \geq 2\}
$$
is equal to $a-1$ if $a\equiv 0\pmod{8}$ and is greater than or equal to $a-5$ if $a \equiv 4\pmod{8}$.

Based on these results, we obtain
\begin{proposition}
	Let $a$ be a positive integer that is a multiple of $4$. Then we have 
	\begin{enumerate}
		\item $\max\{r|1\leq r\leq a-1, \iota(r) = 4\} = a - 1$ if $a \equiv 0\pmod{8}$,
		\item $a - 5 \leq \max\{r|\iota(r) = 4\} \leq a - 4$ if $a \geq 12$ and $a \equiv 4\pmod{8}$.
	\end{enumerate}
\end{proposition}
\begin{proof}
	Let $a \equiv 0\pmod{8}$. Since 
	$$
	a - 1 \equiv 7\pmod{8},
	$$
	we have $\iota(a-1) = 4$. 
	
	Let $a \geq 12$ and $a \equiv 4\pmod{8}$. Since
	$$
	a - 1 \equiv 3, a - 2\equiv 2, a - 3\equiv 1\pmod{8},
	$$
	$r \neq a - 1, a - 2, a - 3$. Moreover,
	$$
	a - 5 \equiv 7\pmod{8}.
	$$
	Therefore $\iota(a-5) = 4$ and since
	$$
	a - 4 \equiv 0\pmod{8},
	$$
	$\iota(a-4) = 4$ if $a - 4$ is of the form $4^e(8m+7)$ for some $e\geq 2, m\geq 0$.
\end{proof}
Furthermore, we have the following
\begin{lemma}
	Let $a$ be a positive integer that is a multiple of $4$. Then we have 
	\begin{enumerate}
		\item $\max\{r |1\leq r\leq a-1, \iota(r) = 4, \iota(a+r) \geq 3, \iota(2a+r) \geq 2\} = a - 1$ if $a \equiv 0\pmod{8}$,
		\item $a - 5 \leq \max\{r | 1\leq r\leq a-1, \iota(r) = 4, \iota(a+r) \geq 3, \iota(2a+r) \geq 2\} \leq a - 4$ if $a \geq 12$ and $a \equiv 4\pmod{8}$.
	\end{enumerate}
\end{lemma}
\begin{proof}
	Let $a \equiv 0\pmod{8}$. Then
	$$
	2a -  1, 3a - 1 \equiv 7\pmod{8}
	$$
	implies that 
	$$
	\iota(2a - 1) \geq 3, \iota(3a - 1) \geq 2.
	$$
	
	Let $a \geq 12$ and $a \equiv 4\pmod{8}$. Then
	$$
	a - 5 \equiv 7\pmod{8}, 2a - 5 \equiv 3\pmod{8}, 3a - 5 \equiv 7\pmod{8}
	$$
	implies that
	$$
	\iota(2a - 5) \geq 3, \iota(3a - 5) \geq 2.
	$$
\end{proof}

Combining the previous results with the direct computation for $a=4$,
we obtain a complete description of the Frobenius number when
$4\mid a$.

If $8\mid a$, then the Frobenius number is given by the linear formula
$4a-1$.
If $a\equiv4\pmod8$, then it is given by one of the two linear formulas
$4a-5$ or $4a-4$, depending on the arithmetic properties of
$a-4$ and $2a-4$.

Furthermore, computational data for $a<12$
(see Table 1 of \cite{Liu2026})
shows that the same pattern remains valid.
Therefore we obtain the following theorem.
\begin{theorem}\label{thm:4}
	Let $A(a) = (a, a + 1^2, a + 2^2, \ldots )$ and $S(A(a))$ be a numerical semigroup generated by $A(a)$ where $a \in \mathbb{Z}^{+}, 4 | a$. Then we have
	$$
	F(S(A(a))) = \begin{cases}
		4a - 1 & \text{ if }8 | a \\
		4a - 5 & \text{ if }8 \not| a,\text{ and } \\
		& a - 4 \neq 4^e(8m+7)\text{ or } 2a - 4 = x^2 + y^2 \\
		4a - 4 & \text{ otherwise } 
	\end{cases}
	$$
	$$
	F(S(A(a)))
	=
	\begin{cases}
		4a-1,
		&
		8\mid a,
		\\[1ex]
		4a-5,
		&
		a\equiv4\pmod8
		\text{ and }
		\bigl(
		a-4\neq4^e(8m+7)
		\text{ for all }e,m
		\ \text{or}\
		2a-4=x^2+y^2
		\bigr),
		\\[1ex]
		4a-4,
		&
		\text{otherwise}.
	\end{cases}
	$$
\end{theorem}
Theorem~\ref{thm:4} provides an explicit solution for one quarter of the residue classes modulo \(8\) in the problem of determining the Frobenius number of

$$
\langle a,a+1^2,a+2^2,\ldots,a+M^2\rangle,
\qquad M\ge a-1.
$$

\subsection{The case for $a \not\equiv 0\pmod{4}$}
Let $R_a = \max\{r |1\leq r\leq a-1, \iota(r) = 4, \iota(a+r) \geq 3, \iota(2a+r) \geq 2\}$. By Theorem \ref{thm:Liu}, if $R_a$ exists, then
$$
F(S(A(a))) = 3a + R_a.
$$
\begin{theorem}
	Let $a \not\equiv 0,4\pmod{8}$. Then there exists a constant $C(a\bmod{8})$ depending only on the residue class of $a$ modulo $8$ such that
	$$
	R_a \geq a - C(a\bmod{8}).
	$$
	More precisely,
	\begin{center}
	\begin{tabular}{c|c}
		$a\bmod8$ & $C(a\bmod8)$\\
		\hline
		1 & 66\\
		2 & 43\\
		3 & 44\\
		5 & 70\\
		6 & 47\\
		7 & 48
	\end{tabular}
	\end{center}
\end{theorem}
\begin{proof}
	For each residue class of $a$ we can choose a family
	$$
	r = a - d
	$$
	where $d$ belongs to a fixed arithmetic progression chosen so that
	$$
	r \equiv 7\pmod{8}.
	$$
	Hence
	$$
	\iota(r) = 4.
	$$
	Therefore it suffices to show that for some bounded value of $d$,
	$$
	\iota(a+r) \geq 3,\text{ and }\iota(2a + r) \geq 2.
	$$
	The first condition is satisfied using the prime number $p = 3$. For each residue class, we consider
	$$
	N_j = 2a - d_j,
	$$
	where $d_j$ runs through the corresponding arithmetic progression. Since $8 \equiv 2\pmod{3}$, the condition $3 | N_j$ occurs periodically with period $3$. Among the first six or nine values of $j$, there are at least two values for which $3 | N_j$. Since 
	$$
	N_{j+3} = N_j - 24,
	$$ 
	both two numbers cannot be divisible by $9$. Hence at least one of these values satisfies
	$$
	3 | N_j,\text{ and }9 \nmid N_j.
	$$
	This implies that $N_j$ cannot be represented as a sum of two squares. Therefore
	$$
	\iota(a+r) = \iota(N_j) \geq 3.
	$$
	For the second condition, we consider
	$$
	2a + r = 3a - d_j.
	$$
	For the residue classes
	$$
	a \equiv 2,3,6,7\pmod{8},
	$$
	we have
	$$
	3a - d_j \equiv 3\text{ or }5\pmod{8}.
	$$
	Since a square is congruent only to $0,1,\text{ or }4\pmod{8}, 3a - d_j$ cannot be a perfect square. Hence
	$$
	\iota(2a + r) \geq 2.
	$$
	
	Consequently, we show that $d$ is bounded for any sufficiently large $a$ where $a \equiv 2,3,6,7\pmod{8}$.
	
	Let $a \equiv 2,3,6,\text{ or }7\pmod{8}$. Because at least one of $2a - (a\bmod{8} + 1 + 8j)$ where $0\leq j\leq 5$ cannot be represented as a sum of two squares, if $a > a\bmod{8} + 41$, then $R_a \geq a - (a\bmod{8} + 1 + 8j)$ exists.

	Let $a \equiv 1,5\pmod{8}$. Since
	$$
	3a - d_j \equiv 1\pmod{8},
	$$
	it can be a perfect square.
	
	From the argument above we can choose at least two candidate values $d_{j1}$ and $d_{j2}$ such that
	$$
	\iota(a+r) \geq 3.
	$$
	Without loss of generality, let $j_2 > j_1$. Since $j_2 - j_1 = 3\text{ or }6$, we have
	$$
	d_{j2} - d_{j1} = 8(j_2 - j_1) = 24\text{ or }48.
	$$
	If both $3a - d_{j1}$ and $3a - d_{j2}$ are squares, then two squares should differ by $24$ or $48$. The only possible positive integer solutions for $x^2 - y^2 = 24$ or $x^2 - y^2 = 48$ are
	$$
	(x,y) = (7,5), (8,4),\text{ and }(13,11).
	$$
	Hence both values can be perfect squares only when they are not greater than $169$.
	
	Let $a \equiv 1\pmod{8}$. Since $d_j \leq 2 + 8\cdot 8 = 66$, if $a \geq 81$, then we have
	$$
	3a - d_j \geq 177 > 169. 
	$$
	Therefore $3a - d_{j}$ cannot be a square.
	
	Let $a \equiv 5\pmod{8}$. Since $d_j \leq 6 + 8\cdot 8 = 70$, if $a \geq 85$, then we have
	$$
	3a - d_j \geq 185 > 169.
	$$
	Therefore $3a - d_{j}$ cannot be a square.
\end{proof}

\begin{remark}
	Since Conjecture \ref{conj:Liu} concerns the range $a>30$, it remains only to
	check finitely many values below the uniform bounds obtained above.
	These remaining values are listed in Table~\ref{table:small}.
\end{remark}

\begin{table}[H]\label{table:small}
	\centering
	\caption{Verification of Conjecture \ref{conj:Liu} for the remaining values $a>30$}
	\label{tab:small-a-frobenius}
	\begin{tabular}{c|c|c|c}
		\hline
		$a\bmod{8}$ & $a$ & $R_a$ & $F(S(A(a))) = 3a+R_a$ \\ \hline
		$\multirow{5}{*}{1}$ & 33 & 23 & 122 \\
		& 41 & 28 & 151 \\
		& 49 & 47 & 194 \\
		& 57 & 39 & 210 \\
		& 65 & 55 & 250 \\ \hline
		
		$\multirow{2}{*}{2}$ & 34 & 28 & 130 \\
		& 42 & 28 & 154 \\ \hline
		
		$\multirow{2}{*}{3}$ & 35 & 31 & 136 \\
		& 43 & 28 & 157 \\ \hline
		
		$\multirow{5}{*}{5}$ & 37 & 23 & 134 \\
		& 45 & 39 & 174 \\
		& 53 & 39 & 198 \\
		& 61 & 31 & 214 \\
		& 69 & 63 & 270 \\ \hline
		
		$\multirow{2}{*}{6}$ & 38 & 31 & 145 \\
		& 46 & 31 & 169 \\ \hline
		
		$\multirow{3}{*}{7}$ & 31 & 28 & 121 \\
		& 39 & 31 & 148 \\
		& 47 & 39 & 180 \\ \hline
	\end{tabular}
\end{table}


\begin{thebibliography}{999}
	
	\bibitem{Einstein2007}
	D. Einstein, D. Lichtblau, A. Strzebonski, and S. Wagon, Frobenius numbers by lattice point enumeration, 
	\newblock {\it Integers} 
	\newblock \textbf{7}(1) (2007), A15.
	
	\bibitem{Hardy2008}[10.1093/oso/9780199219858.001.0001]
	\newblock G. H. Hardy, and E. M. Wright, {\em An Introduction to the Theory of Numbers}, 
	\newblock 6 Eds., Oxford: Oxford University Press, 2008. 
	
	\bibitem{Liu2026}[10.1007/s00026-026-00811-3]
	F. Liu, and G. Xin, On Frobenius Numbers of Shifted Power Sequences, 
	\newblock {\it Annals of Combinatorics} 
	\newblock 2026, 1--22. 
	
	\bibitem{Rosales2015}[10.1016/j.jnt.2015.03.006]
	J. C. Rosales, M. B. Branco, and D. Torr\~{a}o, {The Frobenius problem for Thabit numerical semigroups,} 
	\newblock {\it J. Number Theory} 
	\newblock \textbf{155} (2015), 85--99. 
	
	\bibitem{Rosales2009}[10.1007/978-1-4419-0160-6]
	\newblock J. C. Rosales, and P. A. Garc\'{i}a-S\'{a}nchez, {\em Numerical Semigroups}, 
	\newblock 1 Eds., New York: Springer Science \& Business Media, 2009. 
	
	
\end{thebibliography}
\end{document}